\documentclass{amsart}
\usepackage{amsfonts,amssymb,amsmath,amsthm}
\usepackage{url}
\usepackage{enumerate}

\newtheorem{theorem}{Theorem}[section]
\newtheorem{lemma}[theorem]{Lemma}
\newtheorem{proposition}[theorem]{Proposition}
\numberwithin{equation}{section}

\author{ James McKee}
\address{Department of Mathematics\\
Royal Holloway, University of London\\
Egham Hill\\
Egham\\
Surrey TW20 0EX\\
UK}
\email{James.McKee@rhul.ac.uk}
\author{ Chris Smyth}
\address{School of Mathematics and Maxwell Institute for Mathematical Sciences\\
University of Edinburgh\\
Edinburgh EH9 3JZ\\
Scotland, U.K.} \email{C.Smyth@ed.ac.uk} 

\begin{document}

\title[Salem numbers and Pisot numbers]
{Salem numbers and Pisot numbers via interlacing}

\begin{abstract}
We present a general construction of Salem numbers via rational functions whose zeros and poles mostly lie on the unit circle and satisfy an interlacing condition.  This extends and unifies earlier work.  We then consider the `obvious' limit points of the set of Salem numbers produced by our theorems, and show that these are all Pisot numbers, in support of a conjecture of Boyd.  We then show that all Pisot numbers arise in this way.  Combining this with a theorem of Boyd, we  produce all Salem numbers via an interlacing construction.
\end{abstract}

\maketitle

\section{Introduction}

A {\em Pisot number} is a real algebraic integer $\theta>1$,
all of whose other (algebraic) conjugates have modulus strictly
less than 1.
A {\em Salem number} is a real algebraic integer $\tau>1$,
whose other conjugates all have modulus at most 1, with at least
one having modulus exactly 1.
It follows that the minimal polynomial $P(z)$ of $\tau$ is
{\em reciprocal} (i.e., $z^{{\rm deg\ }P}P(1/z)=P(z)$), that
$\tau^{-1}$ is a conjugate of $\tau$, that all conjugates of
$\tau$ other than $\tau$ and $\tau^{-1}$ have modulus exactly 1,
and that $P(z)$ has even degree.
The set of all Pisot numbers is traditionally
denoted $S$, with $T$ being used for the set of all
Salem numbers.

In \cite{MS3}, we constructed Salem numbers via rational functions associated to certain rooted trees (the \emph{quotients} of rooted \emph{Salem trees}).  In this paper we abstract the essential properties of these rational functions, and give a much more general construction of Salem numbers (Theorems \ref{T:CCSalem}, \ref{T:CSSalem} and \ref{T:SSSalem}) via rational functions whose zeros and poles mostly lie on the unit circle and satisfy an interlacing condition.  In addition to extending the work of \cite{MS3}, this also extends the interlacing construction of \cite {MS2}.  We then consider the `obvious' limit points of the set of Salem numbers produced by our theorems, and show that these are all Pisot numbers (Theorems \ref{T:CCPisot} and \ref{T:SSPisot}).  This supports a conjecture of Boyd \cite[p.~327]{Bo}.  We then show that all Pisot numbers arise in this way (Theorem \ref{T:allPisots}).  Combining this with a theorem of Boyd, we show that all Salem numbers can be produced via interlacing.  We conclude the paper with some applications to the study of small Salem numbers and negative-trace elements of $S$ or $T$.

It is our hope that these ideas will lead to further
improvements in our understanding of the set of Salem numbers, and
may give a way to attack some outstanding problems: (i) is there a
least Salem number, and, if so, what is it? (ii) is the set of Salem
numbers below (say) $1.3$ finite, and, if so, what are its members?
(iii) what is the derived set of the set of Salem numbers?

For dealing with Pisot numbers one has the trivial but extremely
useful observation that if $f(z)$ is a monic polynomial with
integer coefficients having a simple real root $\theta>1$, such
that all roots other than $\theta$ have modulus strictly less than
1, and the constant term of $f(z)$ is not 0, then $f(z)$ is
irreducible, and is therefore the minimal polynomial of $\theta$
(if $f(z)$ split into two nontrivial factors, then the factor
that does not have $\theta$ as a root would have as its constant
term something that on the one hand is a nonzero integer, and on
the other hand is a product of numbers all with modulus strictly
less than one, which is absurd). For Salem numbers, the analogous
statement is not as pleasant: if $g(z)$ is a monic polynomial with
integer coefficients, having a simple real root $\tau>1$, such
that all the other roots of $g(z)$ have modulus at most one, with
at least one having modulus equal to 1, and if the constant term of
$g(z)$ is not zero, then $g(z)=t(z)u(z)$, where $t(z)$ is the
minimal polynomial of $\tau$, and $u(z)$ is a \emph{cyclotomic
polynomial} (for us, following \cite{Bo}, meaning simply that all its roots are roots of unity:
it need not be irreducible). It is the possibility that $u(z)$
might not equal 1 that renders explicit constructions of the
minimal polynomials of Salem numbers more difficult. For Pisot
numbers it is enough to find a polynomial that has all its roots
in the right place; for Salem numbers one also has to deal with the possibility of
cyclotomic factors. A further annoyance is that $t(z)$ might have
degree 2, in which case one has that $\tau$ is a reciprocal quadratic Pisot
number rather than a Salem number.

With these thoughts in mind, it is convenient to define a {\em
Pisot polynomial} to be a polynomial of the form $z^kf(z)$, where
$k\ge 0$ and $f(z)$ is the minimal polynomial of a Pisot number.
And we define a {\em Salem polynomial} to be a polynomial of the
form $t(z)u(z)$, where $u(z)$ is a cyclotomic polynomial, and
$t(z)$ is either the minimal polynomial of a Salem number or is
the minimal polynomial of reciprocal quadratic Pisot number.

The plan for the remainder of the paper is as follows.
In \S\ref{S:interlacing} we define the various interlacing conditions that will subsequently be exploited.
Section \ref{S:CCSalems} shows how Salem numbers can be produced from pairs of polynomials that satisfy a simple circular interlacing condition; then  \S\ref{S:CCPisots} considers the obvious limit points of the set of Salem numbers produced, and shows that these are all Pisot numbers.
In  \S\ref{S:CSSSconstructions} we prove analogous results for other naturally-arising variants of interlacing.
In \S\ref{S:allPisots} we show that all Pisot numbers are generated by one of these interlacing constructions, and in  \S\ref{S:allSalems} we show that all Salem numbers are produced, and we put Salem numbers into four (overlapping) subsets according to the flavour of interlacing used to produce them.

Several other interlacing constructions appear in the literature.
Most notably, Bertin and Boyd \cite{BB} classify all Salem
numbers in a way that involves interlacing. In
\S\ref{S:BB} we briefly compare their results with ours, before some concluding applications and remarks in  \S\ref{S:conclusion}.
Other
interlacing constructions have appeared in \cite{CW} (Proposition
4.1), \cite{L}, and \cite{MS2}. For an encyclopaedic account of
real interlacing, see \cite{F}.

We use $\mathbb{T}$ to denote the unit circle, $\mathbb{T}=\bigl\{ z \in \mathbb{C} \,\big|\, |z|=1 \bigr\}$.

\section{Flavours of interlacing}\label{S:interlacing}
Several variants of interlacing will be seen to arise naturally as we study Salem numbers.
We are concerned with interlacing on the unit circle, but the different flavours of interlacing are perhaps most easily understood when one moves to the real line via a Tchebyshev transformation.
In  \S\ref{SS:Tch} we recall this transformation; in  \S\S\ref{SS:CC}--\ref{SS:SS} we describe interlacing in the complex world, and in \S\ref{SS:realiq} we view it from the real, post-Tchebyshev, perspective.

\subsection{Moving to the real world}\label{SS:Tch}
Our ultimate objective is to understand Salem numbers and Pisot
numbers, and these are firmly rooted in the world of complex
numbers. We shall give constructions that involve {reciprocal}
polynomials.  Moreover most (perhaps all) of their roots will be
in $\mathbb{T}$, and other roots will be real and positive. It
will be extremely convenient for the proofs to transform such
polynomials to totally real polynomials.  The transformation that
we shall use is
\begin{equation}\label{E:xfromz}
x = \sqrt{z} + 1/\sqrt{z}\,.
\end{equation}
It is a matter of historical accident (growing out of \cite{MRS},
where this particular transformation was essential) that this
variant of the  Tchebyshev transformation is used rather than the
more familiar $x = z + 1/z$, which would serve just as well, but
with many small differences in detail. In applying
(\ref{E:xfromz}), a fixed branch of the square-root is used
throughout the right-hand side, but since there is a choice of
branch we generally find two possible values of $x$. If $z\in\mathbb{T}$, or if $z$ is real, then the corresponding one or two
values of $x$ are real.

The transformation (\ref{E:xfromz}) is generally a 2-to-2 map,
with a reciprocal pair $z$, $1/z$ mapping to a pair $x$, $-x$. The
exceptions are important for us: the single point $z=-1$
corresponds to the single point $x=0$, and the single point $z=1$
corresponds to the pair $x=2$, $x=-2$.  The inverse correspondence
involves solving a quadratic equation, but we shall never have
need for it explicitly.

\subsection{CC-interlacing}\label{SS:CC}
Suppose that $P(z)$ and $Q(z)$ are coprime polynomials with integer
coefficients, and with positive top coefficients.  We say that $Q$
and $P$ satisfy the \emph{CC-interlacing condition}, or that $Q/P$
is a \emph{CC-interlacing quotient} if:
\begin{itemize}
\item $P$ and $Q$ have all their roots in $\mathbb{T}$;

\item all their roots are simple;

\item their roots interlace on the unit circle, in the sense that
between every pair of roots of $P(z)$ there is a root of $Q(z)$,
and between every pair of roots of $Q(z)$ there is a root of
$P(z)$.
\end{itemize}
Extending to real coefficients, one recovers the \emph{circular
interlacing condition} of \cite{MS2}.  If $P$ and $Q$ satisfy the
CC-interlacing condition, then they must have the same degree.
Moreover, both $1$ and $-1$ must appear among their roots. One of
$P$ and $Q$ is a reciprocal polynomial; the other is antireciprocal ($z-1$ times a
reciprocal polynomial).  The nomenclature is a shorthand for
``cyclotomic-cyclotomic interlacing'', which in turn is a slight
abuse of terminology: the two polynomials have all their roots in $\mathbb{T}$, but need not be cyclotomic since they need not be
monic.

For an example (derived from the quotient attached to $\tilde{E}_8(8)$ in \cite[p.~220]{MS3}) to which we shall return later, take
\begin{equation}\label{E:Lehmer}
\begin{array}{rcl}
P(z) &=& (z-1)(z+1)(z^2 + z+ 1)(z^4 + z^3 + z^2 + z+1)\,, \\
\\
Q(z) &=& z^8 + z^7 - z^5 - z^4 -z^3 + z + 1\,.
\end{array}
\end{equation}
Thus $Q(z)$ is the thirtieth cyclotomic polynomial, and $P(z)$ is
the product of the first, second, third and fifth cyclotomic
polynomials.  The roots of $P$ and $Q$ interlace on the unit
circle, as shown in Figure \ref{F:CC}: $Q/P$ is a CC-interlacing
quotient.

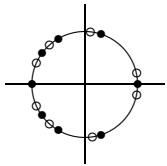
\begin{figure}

\begin{picture}(60,60)(-30,-20)
\put(-30,0){\line(1,0){60}} \put(0,-30){\line(0,1){60}}
\put(0,0){\circle{40}} \put(20,0){\circle*{3}}
\put(-20,0){\circle*{3}} \put(-10,17.3){\circle*{3}}
\put(-10,-17.3){\circle*{3}} \put(6.2,19){\circle*{3}}
\put(6.2,-19){\circle*{3}} \put(-16.2,11.8){\circle*{3}}
\put(-16.2,-11.8){\circle*{3}} \put(19.6,4.2){\circle{3}}
\put(19.6,-4.2){\circle{3}} \put(2.1,19.9){\circle{3}}
\put(-13.4,14.9){\circle{3}} \put(-18.3,8.1){\circle{3}}
\put(-18.3,-8.1){\circle{3}} \put(-13.4,-14.9){\circle{3}}
\put(2.9,-19.9){\circle{3}} 
\end{picture}
\caption{CC-interlacing. The roots of $(z-1)(z+1)(z^2+z+1)(z^4+z^3+z^2+z+1)$ [$\bullet$]
 interlace on $\mathbb T$ with those of $z^8+z^7-z^5-z^4-z^3+z+1$ [$\circ$].}
\label{F:CC}
\end{figure}

Our definition is symmetric in $P$ and $Q$: if $Q/P$ is a CC-interlacing quotient, then so is $P/Q$.

Note that the definition of the CC-interlacing condition does not
require either $P$ or $Q$ to be monic.  When both are monic, then
by a theorem of Kronecker \cite{Kro} they are cyclotomic.  In this
case, all interlacing examples have essentially been classified by Beukers and
Heckman \cite{BH}.

\subsection{CS-interlacing}
Now we turn to another flavour of interlacing, where one
polynomial has all its roots in $\mathbb{T}$, and the other has
all but two roots in $\mathbb{T}$, with these two roots
being $\theta$ and $1/\theta$ for some real $\theta>1$.  Here
``CS'' is short for ``cyclotomic-Salem'', with the same caveat as
before that the polynomials need not be monic.  One will be
reciprocal, and the other will be antireciprocal.

Suppose that $P(z)$ and $Q(z)$ are coprime polynomials with integer
coefficients, and with positive top coefficients.  We say that $P$
and $Q$ satisfy the \emph{CS-interlacing condition} and
that $Q/P$ is a \emph{CS-interlacing quotient} if:
\begin{itemize}
\item $P$ is reciprocal, and $Q$ is antireciprocal;

\item $P$ and $Q$ have the same degree;

\item all the roots of $P$ and $Q$ are simple, except perhaps at $z=1$;

\item $z^2 - 1 \mid Q$;

\item $Q$ has all its roots in $\mathbb{T}$;

\item $P$ has all but two roots in $\mathbb{T}$, with these
two being real, positive and $\ne 1$;

\item on the punctured unit circle $\mathbb{T}\backslash\{1\}$, the roots of $Q$ and
$P$ interlace.
\end{itemize}
Notice the strange interlacing condition.  On the unit circle, $Q$ has two more roots than $P$, and necessarily $Q(1)=0$.  The interlacing condition implies that either $Q$ has a triple root at $1$, or it has a pair of simple roots that are closer to $1$ on the unit circle than any of the roots of $P$.

A couple of pictures should clarify this: Figures \ref{F:CSsimple} and \ref{F:CStriple}.

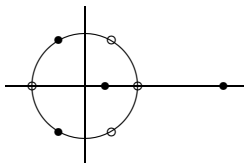
\begin{figure}
\begin{picture}(60,60)(-30,-20)
\put(-30,0){\line(1,0){90}} \put(0,-30){\line(0,1){60}}
\put(0,0){\circle{40}}

\put(-20,0){\circle{3}} \put(20,0){\circle{3}}
\put(10,17.3){\circle{3}} \put(10,-17.3){\circle{3}}

\put(-10,17.3){\circle*{3}} \put(-10,-17.3){\circle*{3}} \put(7.6,0){\circle*{3}} \put(52.4,0){\circle*{3}}

\end{picture}
\caption{{CS-interlacing with simple roots.} {The roots of
$Q=(z^2-1)(z^2-z+1)$ [$\circ$] interlace on $\mathbb{T}\backslash\{1\}$}
 {with those of $P=(z^2+z+1)(z^2-3z+1)$
[$\bullet$].}}
\label{F:CSsimple}
\end{figure}

\begin{figure}
\begin{picture}(60,60)(-30,-20)
\put(-30,0){\line(1,0){90}} \put(0,-30){\line(0,1){60}}
\put(0,0){\circle{40}}

\put(-20,0){\circle{3}} \put(20,0){\circle{3}} \put(20,0){\circle{5}} \put(20,0){\circle{7}}

\put(-10,17.3){\circle*{3}} \put(-10,-17.3){\circle*{3}} \put(7.6,0){\circle*{3}} \put(52.4,0){\circle*{3}}

\end{picture}
\caption{{CS-interlacing with a triple root at $1$.} {The roots of
$Q=(z+1)(z-1)^3$ [$\circ$] interlace on $\mathbb{T}\backslash\{1\}$}
{with those of $P=(z^2+z+1)(z^2-3z+1)$
[$\bullet$].}}
\label{F:CStriple}
\end{figure}
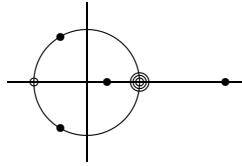

There is no symmetry in the CS-interlacing conditions: if $Q/P$ is a CS-interlacing quotient, then $P/Q$ is not.

\subsection{SS-interlacing}\label{SS:SS}
For our third flavour of interlacing, ``SS'' suggests ``Salem-Salem'' with the usual caveats.

Suppose that $P(z)$ and $Q(z)$ are coprime polynomials with integer
coefficients, and with positive top coefficients.  We say that $P$
and $Q$ satisfy the \emph{SS-interlacing condition} and
that $Q/P$ is an \emph{SS-interlacing quotient} if:
\begin{itemize}
\item $P$ and $Q$ have the same degree;

\item all the roots of $P$ and $Q$ are simple;

\item one of $P$ and $Q$ is reciprocal, the other is
antireciprocal;

\item each of $P$ and $Q$ has all but two of its roots in $\mathbb{T}$, with these two being real, positive and $\ne 1$;

\item on the unit circle, the roots of $Q(z)$ and $P(z)$
interlace.
\end{itemize}

The behaviour of the real roots of $P$ and $Q$ gives us two possible types of SS-interlacing.  If $Q/P$ is an SS-interlacing quotient then we say that it is a \emph{type $1$ interlacing quotient} if the largest real root of $PQ$ is a root of $P$, and it is a \emph{type $2$ interlacing quotient} if the largest real root of $PQ$ is a root of $Q$.  There is symmetry in the conditions for SS-interlacing, but between  the two types: $Q/P$ is a type $1$ SS-interlacing quotient if and only if $P/Q$ is a type $2$ SS-interlacing quotient.
Again it is helpful to see a picture: Figure \ref{F:SS}.

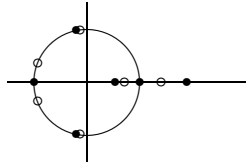
\begin{figure}
\begin{picture}(60,60)(-30,-20)
\put(-30,0){\line(1,0){90}} \put(0,-30){\line(0,1){60}}
\put(0,0){\circle{40}}

\put(14.27,0){\circle{3}} \put(28.03,0){\circle{3}} \put(-18.6,-7.3){\circle{3}} \put(-18.6,7.3){\circle{3}} \put(-2.54,19.8){\circle{3}} \put(-2.54,-19.8){\circle{3}}

\put(-20,0){\circle*{3}} \put(20,0){\circle*{3}} \put(10.6,0){\circle*{3}} \put(37.7,0){\circle*{3}} \put(-4.14,19.6){\circle*{3}} \put(-4.14,-19.6){\circle*{3}}

\end{picture}
\caption{{Type $1$ SS-interlacing.} {The roots of
$Q=z^6-z^4-z^3-z^2+1$ [$\circ$] interlace on $\mathbb{T}$ with} {those of $P=z^6-2z^5+2z-1$
[$\bullet$].} For type 2 SS-interlacing, interchange $P$ and $Q$.}
\label{F:SS}
\end{figure}

Swapping the roles of $P$ and $Q$ in the example in Figure \ref{F:SS} gives an example of type $2$ SS-interlacing.  Notice that we do \emph{not} insist that the roots on the positive real axis interlace (although in this particular example they do).

\subsection{Real interlacing quotients}\label{SS:realiq}
From any of the above flavours and types of interlacing pairs, we
shall consider transforming the pair to a rational function with
only real zeros and poles.  These zeros and poles will generally
interlace (though the interlacing is not always perfect), and for
convenience we shall refer to the rational function as a (real)
interlacing quotient.

If $P$ and $Q$ satisfy the CC-interlacing condition, or the
CS-interlacing condition, or either type of SS-interlacing
condition, then we transform the function $\sqrt{z}Q(z)/(z-1)P(z)$ via the
map (\ref{E:xfromz}) to get a quotient $q(x)/p(x)$, with $q$ and
$p$ coprime polynomials in $\mathbb{Z}[x]$, and $xq(x)/p(x)$ a rational function in $x^2$.  Suppose $P$ and $Q$
have degree $d$.  If $z-1\mid Q(z)$ (which must be the case for
CS-interlacing), then when considering $Q(z)/(z-1)P(z)$ we have
pulled out a root of $Q$, and the remaining roots of $P$ and $Q$
transform in a $2$-to-$2$ or $1$-to-$1$ manner, so that $q$ has
degree $d-1$ and $p$ has degree $d$. If $z-1\mid P(z)$, then the
factor $(z-1)^2$ in the denominator of $Q(z)/(z-1)P(z)$ transforms
to $x^2-4$: we conclude that $q$ has degree $d$ and $p$ has degree
$d+1$.  We call $q(x)/p(x)$ the \emph{(real) interlacing quotient}
corresponding to $Q(z)/P(z)$.  The conditions on the roots of $P$
and $Q$ are sufficient to ensure that the roots of $p$ and $q$ are
all real.

For CC-interlacing, CS-interlacing and type $1$ SS-interlacing the roots
of $p$ and $q$ interlace perfectly: the zeros and poles of the
interlacing quotient interlace.  The quotient $q(x)/p(x)$ is
decreasing wherever it is defined, and has partial fraction
expansion
\begin{equation}\label{E:partialfraction}
\sum_{i=1}^{\deg p} \frac{\lambda_i}{ x - \alpha_i}
\end{equation}
where the $\alpha_i$ are the roots of $p$ and the $\lambda_i$ are
all positive.

For type 2 SS-interlacing, there is perfect interlacing of the
zeros of $q$ and $p$ within the interval $[-2,2]$, but there is a
blip to the right of $x=2$ (and to the left of $x=-2$)
with the top (and bottom) zeros of $p$ and $q$ being in the wrong order for perfect interlacing.
The
derivative of the quotient $q(x)/p(x)$ changes sign twice, 
and the partial fraction expansion (\ref{E:partialfraction}) has
two of the $\lambda_i$ negative.

Note that a real interlacing quotient $q(x)/p(x)$, as defined, is
always an odd function: one of $p$ and $q$ is an even polynomial
and the other is an odd polynomial.  The degree of the denominator is one more than the degree of the numerator, and the top coefficients are positive.
As $x\rightarrow\infty$, $q(x)/p(x)\rightarrow0$ from above.

Proposition 3.3 of \cite{MS2} extends to this setting.

\begin{lemma}\label{L:interlacingsum}\begin{itemize}
\item[(a)] If $Q_1/P_1$ and $Q_2/P_2$ are CC-interlacing
quotients, then so is their sum.

\item[(b)] Suppose that $Q_1/P_1$ is either a CS-interlacing
quotient or an SS-interlacing quotient and that $Q_2/P_2$ is a CC-interlacing quotient.
Then $Q_1/P_1 + Q_2/P_2$ is either a CS-interlacing quotient or an
SS-interlacing quotient.
\end{itemize}
\end{lemma}

\begin{proof}
Part (a) is just Proposition 3.3 of \cite{MS2}.

For (b),  we transform to the real world, where it easy to see
that everything is of the right shape. Let $q_1/p_1$ and $q_2/p_2$
be the corresponding real interlacing quotients. Then the partial
fraction expansions of $q_1/p_1$ and $q_2/p_2$ as in
(\ref{E:partialfraction}) will have all the $\lambda_i$ positive,
except in the case of type $2$ SS-interlacing, when the
$\lambda_i$ corresponding to the largest and smallest $\alpha_i$
are negative: these correspond to the roots of $p_1$ outside
$[-2,2]$.  The sum $q_1/p_1+q_2/p_2$ will be of the same form:
either all the numerators in the partial fraction expansion will
be positive, or there will be precisely two negative numerators
corresponding to the roots of $p_1$ outside $[-2,2]$.  This is the
right shape for CS/SS-interlacing: for type-$2$ SS-interlacing we
know that $q_1(x)/p_1(x)\rightarrow -\infty$ as $x$ approaches the largest pole from above, so the same is true for the sum. Also, both $q_1(x)/p_1(x)$ and $q_2(x)/p_2(x)$ are positive for all sufficiently large $x$, so the sum has a zero to the right of this pole.
\end{proof}

\section{Salem numbers via CC-interlacing}\label{S:CCSalems}
We now show how to produce Salem numbers from CC-interlacing quotients.
The first construction, which is essentially that of \cite{MS2}, uses a single quotient; we then consider a product construction combining two interlacing quotients in a multiplicative manner, inspired by (but greatly generalising) a formula for the quotients of certain Salem trees \cite{MS3}.

\subsection{A single pair}
Our first interlacing construction is a translation of Proposition
3.2(a) of \cite{MS2}.  This is also a special case of our second
construction, Theorem \ref{T:CCproductSalem}.

\begin{theorem}\label{T:CCSalem}
Let $Q/P$ be a CC-interlacing quotient, with the additional
constraint that $P$ is monic. Let $q/p$ be the corresponding real
interlacing quotient. If
\begin{equation}\label{E:limineq}
\lim_{x\rightarrow 2+}q(x)/p(x)>2\,,
\end{equation}
then the only solutions to the equation
\begin{equation}\label{E:CCSalem}
\frac{Q(z)}{(z-1)P(z)} = 1 + \frac{1}{z}
\end{equation}
are a Salem number (or a reciprocal quadratic Pisot number), its
conjugates, and possibly one or more roots of unity.
\end{theorem}

This is proved in \cite{MS2} using the transformation $x=z+1/z$.
It also follows from Theorem \ref{T:CCproductSalem} on taking
$P_1=P$, $Q_1=Q$, $P_2 = z+1$, $Q_2 = z-1$.  Nevertheless we give
a proof here, using the transformation $x=\sqrt{z}+1/\sqrt{z}$, as this
provides a model for later generalisations.

\begin{proof}
Suppose $P$ and $Q$ have degree $d$.  Since the real interlacing
quotient $q/p$ is decreasing (except for jumps at poles), the equation $q(x)/p(x)=x$ has
exactly one (simple) root between each pair of consecutive roots
of $p$ (these all lie in the interval $[-2,2]$).  The condition
(\ref{E:limineq}) implies the existence of exactly one solution to
$q(x)/p(x)=x$ in the interval $(2,\infty)$.  We have now accounted
for all the roots of $xp(x)-q(x)$, which is a monic polynomial
(given that $P$ is monic) of degree $d+1$ or $d+2$ (according as
$z-1\mid Q$ or $z-1\mid P$). Transforming back to the complex
world, we see that all but two of the solutions to
(\ref{E:CCSalem}) lie in $\mathbb{T}$, and these two are a
reciprocal pair $\{\tau,1/\tau\}$ with $\tau>1$. Clearing
denominators in (\ref{E:CCSalem}) gives a monic polynomial with
integer coefficients, and degree $d+1$ or $d+2$ as appropriate, so
we are done.
\end{proof}

The condition on $q/p$ at $x=2$ translates to
$\lim_{z\rightarrow1+} Q(z)/(z-1)P(z)>2$, which amounts to either
$P(1)=0$ or ($Q(1)=0$ and) $Q'(1)>2P(1)$.  Thus this condition can
be checked readily without computing $q$ and $p$.

For an example, take $P$ and $Q$ as in (\ref{E:Lehmer}).  We have
CC-interlacing, and also $P(1)=0$, and $P$ is monic.  Solving
(\ref{E:CCSalem}) gives the famous Lehmer polynomial $z^{10}+z^9 -
z^7 - z^6 - z^5 - z^4 - z^3 + z + 1$.

To see that cyclotomic factors may appear, consider $P(z) = z^{10}+z^7-z^3-1$, $Q(z) = 2z^{10} + z^8 +
2z^7 + z^6 + 2z^5 + z^4 + 2z^3 + z^2 + 2$.  Again we have $P$
monic and $P(1)=0$.  Now (\ref{E:CCSalem}) gives the four primitive
eighth roots of unity as solutions, as well as the degree-8 Salem
number with minimal polynomial $z^8 - 2z^7 - z^6 - 3z^4 - z^2 - 2z
+ 1$.

\subsection{A product construction}
The following extension of Theorem \ref{T:CCSalem} exploits two
CC-interlacing pairs $(P_1,Q_1)$ and $(P_2,Q_2)$.  Of course, after Lemma \ref{L:interlacingsum}, one
possible way of combining two such pairs is to write $P_1/Q_1 +
P_2/Q_2 = P_3/Q_3$, giving a third pair $(P_3,Q_3)$ that could be
used in Theorem \ref{T:CCSalem}.   Instead of the sum, we consider now the
product: this will no longer give CC-interlacing, but we can
still squeeze out Salem numbers.

\begin{theorem}\label{T:CCproductSalem}
Let $Q_1/P_1$ and $Q_2/P_2$ be two CC-interlacing quotients, with
$P_1$ and $P_2$ both monic.  Let $q_1/p_1$ and $q_2/p_2$ be the
corresponding real interlacing quotients.

(i) Suppose that
\[
\lim_{x\rightarrow 2+} \left(\frac{q_1(x)}{p_1(x)}-2\right)\left(\frac{q_2(x)}{p_2(x)}-2\right) < 1\,.
\]
Then the only solutions to  the equation
\[
\left(\frac{Q_1(z)}{(z-1)P_1(z)}-1-\frac{1}{z}\right)
\left(\frac{Q_2(z)}{(z-1)P_2(z)} - 1-\frac{1}{z}\right) = \frac{1}{z}
\]
are a Salem number (or a reciprocal quadratic Pisot number), its
conjugates, and possibly one or more roots of unity.

(ii) Suppose that
\[
\lim_{x\rightarrow 2+} \frac{q_1(x)q_2(x)}{p_1(x)p_2(x)}> 1\,.
\]
Then the only solutions to  the equation
\[
\frac{Q_1(z)Q_2(z)}{(z-1)^2P_1(z)P_2(z)} = \frac{1}{z}
\]
are a Salem number (or a reciprocal quadratic Pisot number), its
conjugates, and possibly one or more roots of unity.

\end{theorem}

Part (i) extends an explicit formula arising from a certain family of Salem trees \cite[Lemma 7.1(ii)]{MS3}.
The proof makes use of the following
lemma.

\begin{lemma}\label{L:product}
Let $\psi_1(x)$ and $\psi_2(x)$ be rational functions in
$\mathbb{Z}(x)$, strictly decreasing on the real line (over
intervals for which they are defined), with simple zeros and
poles.  Write $\psi_1(x)\psi_2(x)=f(x)/g(x)$, where $f(x)$ and
$g(x)$ are coprime polynomials with integer coefficients.  Suppose
that $g(x)$ has real zeros at $a$ and $b$ (with $a<b$).  Then,
counted with multiplicity, the number of solutions to the equation
$\psi_1(x)\psi_2(x)=c$ for $x\in(a,b)$ is independent of real
$c\ge0$.
\end{lemma}

It will be evident from the proof that all relevant solutions to
$\psi_1(x)\psi_2(x)=c$ are simple, except perhaps when $c=0$. It
is possible that $\psi_1\psi_2$ has one or more double zeros, but
it cannot have zeros of higher order. The application of interest
to us will use only that the number of solutions when $c=1$ is the
same as when $c=0$.

\begin{proof}
In intervals where $\psi_1\psi_2$ is positive, it is strictly
monotonic: it is decreasing if both $\psi_1$ and $\psi_2$ are
positive, and it is increasing if both are negative.  As $x$
passes through a zero $x=\alpha$ of $\psi_1\psi_2$, the function
either decreases from $\infty$ to $0$ as $x$ approaches $\alpha$
from below, or $\psi_1\psi_2$ increases from $0$ to $\infty$ as
$x$ increases from $\alpha$ (or both, in which case $\psi_1\psi_2$
has a double zero at $\alpha$: note that if both $\psi_1$ and
$\psi_2$ vanish at $\alpha$ then necessarily both have the same
sign in a punctured neighbourhood of $\alpha$). For any $c\ge0$ it follows that
between any successive poles of $\psi_1\psi_2$ the number of
solutions to $\psi_1(x)\psi_2(x)=c$ is independent of $c$. The
result follows.
\end{proof}

The proof of Theorem \ref{T:CCproductSalem} now follows.  We
take for $\psi_1$ and $\psi_2$ the rational functions
$q_1/p_1-ax$ and $q_2/p_2-ax$, where $a=1$ for part (i) and $a=0$ for part (ii).
These are decreasing where defined, since the $q_i/p_i$ are real interlacing quotients corresponding to CC-interlacing quotients.
Write $\psi_1(x)\psi_2(x)=f(x)/g(x)$,
after cancelling any common factors, so that $f$ and $g$ are
coprime polynomials with integer coefficients. Note that, from the remarks in Section \ref{SS:realiq},
$f(x)/g(x)$ is an even function.  The number of zeros of
$\psi_1\psi_2$ between its extreme poles is equal to the degree of
$f(x)$, since all roots are real. By Lemma \ref{L:product}, this equals the
number of solutions to $f(x)/g(x)=1$: all of these lie in the
interval $[-2,2]$. For part (i), $f/g\sim x^2\rightarrow\infty$ as $x\rightarrow\infty$; for part (ii), $f/g$ tends to a finite non-positive number as $x\rightarrow\infty$. The condition at $x=2$ ensures a solution to
$f(x)/g(x)=1$ in the interval $(2,\infty)$, and by evenness also
in $(-\infty,-2)$. Since $g(x)-f(x)$ is monic, and we have
accounted for all its roots, we are done when we transform back to
the complex world. As before, the condition at $x=2$ transforms to
an easily-checked condition at $z=1$.

\section{Pisot numbers via CC-interlacing}\label{S:CCPisots}
We now construct Pisot numbers by taking limits of convergent
sequences of Salem numbers.  There is a conjecture of Boyd
\cite[p.~327]{Bo} which, if true, would imply that this process will
always yield either a Salem number or a Pisot number.  Our results
in this paper give a confirmation of this conjecture for all the
cases considered.  In this section we consider CC-interlacing;
subsequently (\S\ref{S:CSSSconstructions}) we shall treat briefly the other flavours of
interlacing.

\subsection{CC-limit functions}

We define a \emph{CC-limit function} to be a rational function
$h(z)$ such that there is a sequence of CC-interlacing quotients
$(h_n(z))$ for which $h_n(z)/(z-1)$ converges to $h(z)$ uniformly
in any compact subset of the exterior of the unit disc. For
example, $1/z$ is a CC-limit function, as we could take $h_n(z) =
(z^n-1)(z-1)/(z^{n+1}-1)$; indeed in this case we have uniform
convergence in the set $|z|\ge 1+\varepsilon$, for any
$\varepsilon>0$.

\begin{lemma}\label{L:pisotquotients}
Take any non-negative integers $A$, $r_1$, $r_2$, $r_3$, $r_4$,
not all zero, and positive integers $A_i$, $a_i$ ($1\le i\le
r_1$), $B_i$, $b_i$ ($1\le i\le r_2$), $C_i$, $c_i$ ($1\le i\le
r_3$), $D_i$, $d_i$ ($1\le i\le r_4$). Then the rational function
\begin{equation}\label{E:pisotquotient}
\begin{array}{ll}\displaystyle \frac{A}{z-1} &\displaystyle +
\sum_{i=1}^{r_1}\frac{A_i(z^{a_i}-1)}{(z-1)z^{a_i}} +
\sum_{i=1}^{r_2}\frac{B_iz^{b_i}}{(z-1)(z^{b_i}-1)} \\ \\ &\displaystyle +
\sum_{i=1}^{r_3}\frac{C_i(z^{c_i}+1)}{(z-1)z^{c_i}} +
\sum_{i=1}^{r_4}\frac{D_iz^{d_i}}{(z-1)(z^{d_i}+1)}
\end{array}
\end{equation}
is a CC-limit function.
\end{lemma}

\begin{proof}
Using the Beukers-Heckman classification \cite{BH} (and see also \cite{MS3}, where all these terms (or their reciprocals) appear as quotients of graphs (multiplied by $z-1$)) for interlacing
cyclotomic polynomials, and Lemma \ref{L:interlacingsum}(a), we
can define for each natural number $n$ a CC-interlacing quotient
$Q_n/P_n$ by
\begin{equation}\label{E:CCapprox}
\begin{array}{rcl}
\frac{Q_n(z)}{P_n(z)} &=&
\frac{A(z^n+1)}{z^n-1} +
\sum_{i=1}^{r_1}\frac{A_i(z^{a_i}-1)(z^n-1)}{z^{n+a_i}-1}
 +
\sum_{i=1}^{r_2}\frac{B_i(z^{n+b_i}-1)}{(z^{b_i}-1)(z^n-1)}\\ \\
&& + \sum_{i=1}^{r_3}\frac{C_i(z^{c_i}+1)(z^n-1)}{z^{n+c_i}+1}
 +
\sum_{i=1}^{r_4}\frac{D_i(z^{n+d_i}+1)}{(z^{d_i}+1)(z^n-1)}\,.
\end{array}
\end{equation}
An easy estimate shows that for any $\varepsilon>0$ the sequence
of functions $Q_n(z)/\bigl((z-1)P_n(z)\bigr)$ converges to the
advertised limit function, uniformly in $|z|\ge1+\varepsilon$.
\end{proof}

We shall call a rational function of the shape
(\ref{E:pisotquotient}) a \emph{special CC-limit function}.  For
these we can exploit their explicit form to prove that certain
limit points of the set of Salem numbers are in fact Pisot
numbers.

\subsection{A single interlacing quotient}
Given a single CC-interlacing quotient, we can take the limiting
form of our Salem number construction, and attempt to prove that
the limit is a Pisot number.

\begin{theorem}\label{T:CCPisot}
Let $Q/P$ be either a CC-interlacing quotient or zero ($Q=0$,
$P=1$), with $P$ monic, and put $g(z)=Q(z)/\bigl((z-1)P(z)\bigr)$.  Let $h(z)$ be
a special CC-limit function, as in (\ref{E:pisotquotient}).  Let
$f(z) = g(z) + h(z) - 1 - 1/z$ (if this has a removable
singularity at $z=0$, then remove it). If
\begin{equation}\label{E:zlimineq}
\lim_{z\rightarrow1+}\bigl(g(z)+h(z)\bigr)>2\,,
\end{equation}
then the only non-zero solutions to $f(z)=0$ are a Pisot number
$\theta$, the conjugates of $\theta$, and possibly some roots of
unity.
\end{theorem}

Before proving this, let us make some remarks.  One possible
choice for $h(z)$ is $1/z$, giving simply $f(z)=g(z)-1$.  The
construction of Pisot numbers in \cite{MRS} is essentially that of
Theorem \ref{T:CCPisot} with $h(z)=k/z$ for some positive integer
$k$; the construction in \cite{MS2} uses
$h(z)=1/(z-1)$, which ensures that the condition
(\ref{E:zlimineq}) is satisfied. An application of Theorem
\ref{T:CCPisot} with the more interesting CC-limit function
$z^7/(z-1)(z^7-1)$ is given in \S\ref{SS:negtracePisot}, where it
is used to produce a Pisot number that has negative trace and
degree only $16$, a new record (for old records, see \cite{MRS},
\cite{McK}, \cite{MS2}).

\begin{proof}
For the special CC-limit function $h(z)$, as in
(\ref{E:pisotquotient}), let $Q_n(z)/P_n(z)$ be as in
(\ref{E:CCapprox}), and define $f_n(z) = g(z) + Q_n(z)/(z-1)P_n(z)
- 1 - 1/z$.  We have that for $|z|>1$ the function $f(z)$ is the limit
of the sequence $(f_n(z))$, with convergence uniform in compact
subsets of that region. Moreover, from Theorem \ref{T:CCSalem}
each $f_n(z)$ has a unique root $\tau_n$ in the exterior of the
unit disc, at least for all sufficiently large $n$, say $n\ge n_0$
(so that (\ref{E:limineq}) holds).

Note that $f(z)$ has no poles outside the unit disc, and has
finitely many zeros there (it cannot be identically zero, as the condition (\ref{E:zlimineq}) would then fail). The condition near $z=1$ gives
$\lim_{z\rightarrow1+}f(z)>0$, and we plainly have
$\lim_{z\rightarrow+\infty}f(z)=-1$.  So there is at least one
$\theta$ in the real interval $(1,\infty)$ such that
$f(\theta)=0$.

Take any circle, centred on $\theta$, with radius sufficiently
small that it lies outside the unit disc and such that no zeros of
$f$ other than $\theta$ lie in or on the circle.  For all
sufficiently large $n$, the function $f$ dominates $f_n-f$ on this
circle; hence by Rouch\'{e}'s Theorem (assuming also that $n\ge
n_0$) there is exactly one root of $f_n$ in this circle (for all
sufficiently large $n$), and this root must be $\tau_n$.  We
conclude that $\tau_n\rightarrow\theta$ as $n\rightarrow\infty$.
Since this Rouch\'{e} argument could be applied to any root of $f$
outside the unit disc, but the sequence $(\tau_n)$ has at most one
limit, we conclude that $\theta$ is the only root of $f$ outside
the unit disc.

We deduce that $\theta$ is either a Pisot number or a Salem
number, and the theorem will follow if we show that $\theta$ is
not a Salem number.  Suppose for a contradiction that $\theta$ is
a Salem number, and let $z_0$ be a conjugate of $\theta$ that lies
in $\mathbb{T}$.  For a rational function $k(z)$, write
$\tilde{k}(z)=k(1/z)/z$.  Since $f(z)$ has all coefficients real,
$\overline{z_0}=1/z_0$ is also a zero of $f(z)$.  Thus $z_0$ is a
zero of both $f(z) = g(z) + h(z) - 1 - 1/z$ and $\tilde{f}(z) =
g(z) + \tilde{h}(z) - 1 - 1/z$ (using here that $(z-1)g(z)$ is a
CC-interlacing quotient, so that $g(z)$ is a quotient of
reciprocal polynomials).  Thus $z_0$ is a zero of
$h(z)-\tilde{h}(z)$, and by Galois conjugation so is $\theta$.

For the five special cases $h(z)=1/(z-1)$, $(z^a-1)/(z-1)z^a$,
$z^b/(z-1)(z^b-1)$, $(z^c+1)/(z-1)z^c$, $z^d/(z-1)(z^d+1)$ one
checks explicitly that $h(z)-\tilde{h}(z)$ has no roots outside
the unit disc, giving the desired contradiction.  For the general
case, we appeal to Salem's theorem \cite{Sa} that the set of Pisot
numbers is closed.  Write $h(z) = h_0(z) + h_1(z)$, where $h_0(z)$
is a single term in (\ref{E:pisotquotient}), and $h_1(z)$ is the
rest.  Then take $Q_n/P_n$ as in the proof of Lemma
\ref{L:pisotquotients}, for the limit function $h_1$ (rather than
$h$).  Now for each sufficiently large $n$, we can apply our
special result to conclude that the unique root $\theta_n$ of
$g(z) + Q_n(z)/(z-1)P_n(z) + h_0(z)-1-1/z$ outside the unit disc
is a Pisot number.  (Here we use Lemma \ref{L:interlacingsum}(a)
again, to show that $(z-1)g(z) + Q_n(z)/P_n(z)$ is a
CC-interlacing quotient.)  Now another Rouch\'{e} argument shows
that $\theta$ is the limit of the $\theta_n$, so that Salem's
theorem gives that $\theta$ is a Pisot number.
\end{proof}

\subsection{A product of two quotients}

\begin{theorem}
Let $Q_2/P_2$ and $Q_1/P_1$ each be either a CC-interlacing
quotient or zero, with $P_1$ and $P_2$ both monic, and define (for $i=1$, $2$) $g_i(z) =
Q_i(z)/(z-1)P_i(z)$. Let $h_1$ be a special CC-limit function, and
let $h_2$ be either a special CC-limit function or zero.

(i) Suppose
that
\[
\lim_{z\rightarrow 1+} \bigl(g_1(z) + h_1(z) - 1-1/z\bigr)\bigl(g_2(z) + h_2(z)-1-1/z\bigr)<1\,.
\]
Then the only non-zero roots of the rational function
\[
f(z) = \bigl(g_1(z)+h_1(z) - 1-1/z\bigr)\bigl(g_2(z) + h_2(z)-1-1/z\bigr)-1/z
\]
are a certain Pisot number, $\theta$, its conjugates, and perhaps
some roots of unity.

(ii) Suppose
that
\[
\lim_{z\rightarrow 1+} \bigl(g_1(z) + h_1(z) \bigr)\bigl(g_2(z) + h_2(z))\bigr)>1\,.
\]
Then the only non-zero roots of the rational function
\[
f(z) = \bigl(g_1(z)+h_1(z) \bigr)\bigl(g_2(z) + h_2(z)\bigr)-1/z
\]
are a certain Pisot number, $\theta$, its conjugates, and perhaps
some roots of unity.

\end{theorem}

\begin{proof}
The proof is very similar to that of Theorem \ref{T:CCPisot}, so
we merely spell out the differences.  We again use closure of $S$
to reduce to the special case where $h_2(z)=0$ and $h_1(z)$ is one
of the five special functions $1/(z-1)$, $(z^a-1)/(z-1)z^a$,
$z^b/(z-1)(z^b-1)$, $(z^c+1)/(z-1)z^c$, $z^d/(z-1)(z^d+1)$.  Any
Salem number that is a root of $f(z)$ is also a root of
$f(1/z)/z^2$, so is a common root of $\bigl(g_1(z) + h_1(z)-a_1(1+1/z)\bigr)\bigl(g_2(z)-a_2(1+1/z)\bigr) -
1/z$ and $\bigl(g_1(z)+\tilde{h}_1(z)-a_1(1+1/z)\bigr)\bigl(g_2(z)-a_2(1+1/z)\bigr)-1/z$, so is a root of
$\bigl(h_1(z)-\tilde{h}_1(z)\bigr)\bigl(g_2(z)-a_2(1+1/z)\bigr)$.
As before, $h_1-\tilde{h}_1$ has no zeros outside the unit disc.  Here $g_2(z)-a_2(1+1/z)$ has a single zero outside the unit disc, but we see from the definition of $f$ that this cannot be a zero of $f$.
\end{proof}

\section{Salem and Pisot numbers via
CS/SS-interlacing}\label{S:CSSSconstructions}
Several of the results of the previous section extend to obvious analogues for CS- and SS-interlacing quotients.
We record these here briefly.

\subsection{Salem numbers}
The analogue of Theorem \ref{T:CCSalem} for CS-interlacing
quotients is obvious from a sketch of the graph of $q(x)/p(x) $.
Indeed necessarily one has $q(2)\ge0$ and
$p(2)<0$, so that $q(2)/p(2)\le 0$, making a single root of $q(x)/p(x)=x$ in $(2,\infty)$ automatic.

\begin{theorem}\label{T:CSSalem}
Let $Q/P$ be a CS-interlacing quotient, with the
additional constraint that $P$ is monic.  Then the only solutions
to the equation (\ref{E:CCSalem}) are a Salem number (or a
reciprocal quadratic Pisot number), its conjugates, and possibly
one or more roots of unity.
\end{theorem}

For SS-interlacing quotients, a sufficient condition that there
should be a unique solution to $q(x)/p(x)=x$ in the interval
$(2,\infty)$ is that $q(2)/p(2)\le 2$ (type 1) or $q(2)/p(2)<2$
(type 2).  For type $2$ SS-interlacing the stated condition is not always necessary: it might be possible to have $q(2)/p(2)=2$, depending on the derivative of $q/p$ at $x=2$. But it is simpler to restrict to a strong inequality.

\begin{theorem}\label{T:SSSalem}
Let $Q/P$ be an SS-interlacing quotient (of either type), with the
additional constraint that $P$ is monic.  Suppose further that
\begin{equation}\label{E:SalemviaSS}
\lim_{z\rightarrow 1+}\frac{Q(z)}{(z-1)P(z)}<2\,.
\end{equation}
Then the only solutions to the equation (\ref{E:CCSalem}) are a
Salem number (or a reciprocal quadratic Pisot number), its
conjugates, and possibly one or more roots of unity.  For type 1
SS-interlacing, a weak inequality in (\ref{E:SalemviaSS}) would
suffice.
\end{theorem}

There is no analogue of Theorem \ref{T:CCproductSalem}, as the construction would give two roots
outside the unit disc.
\subsection{Pisot numbers}
Certain limiting cases of Theorems \ref{T:CSSalem} and
\ref{T:SSSalem} yield Pisot numbers. Armed with Lemma
\ref{L:interlacingsum}(b), we can give the analogue of Theorem
\ref{T:CCPisot} in this setting.

\begin{theorem}\label{T:SSPisot}
Let $Q/P$ be either a CS-interlacing quotient or an SS-interlacing
quotient, with $P$ monic, and put $g(z)=Q(z)/\bigl((z-1)P(z)\bigr)$. Let $h(z)$ be
a special CC-limit function, as in (\ref{E:pisotquotient}). Let
$f(z) = g(z) + h(z) - 1 - 1/z$ (if this has a removable
singularity at $z=0$, then remove it). If
$$\lim_{z\rightarrow1+}\bigl(g(z)+h(z)\bigr)<2\,,$$
then the only non-zero solutions to $f(z)=0$ are a Pisot number
$\theta$, the conjugates of $\theta$, and possibly some roots of
unity.
\end{theorem}

\begin{proof}
This is much as before, but now $f$, and each $f_n$ in the
sequence of functions converging to $f$, has a pole outside the
unit disc (the same pole for each $f_n$ and for $f$, corresponding
to the Salem zero of $P$). When considering circles centred on
roots of $f$ outside the unit disc, the radii must be sufficiently
small to avoid enclosing this pole.
\end{proof}

\section{All Pisot numbers via interlacing}\label{S:allPisots}

In this section we shall show that all Pisot numbers are produced
by a special case of Theorem \ref{T:SSPisot} (Theorem
\ref{T:allPisots} below).  We proceed in three steps: in
\S\ref{SS:Pk} we define a sequence of polynomials
$(P_k)_{k\ge0}$, following Salem; in \S\ref{SS:Winding} we
show that for all sufficiently large $k$ the pair $(P_k,P_{k+1})$ is an SS-interlacing quotient (that these polynomials are Salem polynomials is contained in Salem's work---the novelty here is in establishing the interlacing property); in \S\ref{SS:allPisots} we tie
everything together to produce Theorem \ref{T:allPisots}.

For any polynomial $A(z)\in \mathbb{Z}[z]$ of exact degree $d$, define $A^*(z) = z^dA(1/z)$.

\subsection{The polynomials $P_k$}\label{SS:Pk}
\begin{lemma}\label{L:Pk}
Let $A(z)$ be any polynomial of degree $d$ with integer
coefficients. 
For $k\ge0$, define
$P_k(z) = \bigl(z^kA(z) - A^*(z)\bigr)/(z-1)$\,.  Then for $k\ge0$ we have \begin{equation}\label{E:A}
z^kA(z) = P_{k+1}(z) - P_k(z)\,.
\end{equation}
If $k\ge 1$, then the polynomial $P_k$ has degree $d+k-1$.  If $k\ge 1$, then $P_k$ is a reciprocal polynomial; $P_0$ is a power of $z$ times a reciprocal polynomial.  The polynomials $P_k$ satisfy the recurrence
\begin{equation}\label{E:recurrence}
P_{k+2}-(z+1)P_{k+1} + zP_k  = 0\,,
\end{equation}
for $k\ge 0$.
For each $k\ge1$, the pair of polynomials $(P,Q)=(P_{k+1},P_{k})$ is the unique pair of reciprocal polynomials such that the degrees of $P$ and $Q$ are $d+k$ and $d+k-1$ and such that $z^kA(z) = P(z) - Q(z)$.
\end{lemma}

\begin{proof}
This is a collection of simple assertions, all of which follow directly from the definitions.  For the recurrence, its characteristic  polynomial is $X^2 - (z+1)X + z = (X-z)(X-1)$.
\end{proof}


Suppose that $A(z)$ is monic.  If $A(0)\ne 1$, then the degree of $P_0$ is $d-1$, but if $A(0)=1$ then this degree is at most $d-2$.
We record as a lemma the observation that no further cancellation in the degree of $P_0$ can occur if $A(z)$ is the minimal polynomial of a Pisot number $\theta$, unless $\theta$ is a reciprocal quadratic Pisot number.

\begin{lemma}\label{L:nonpal}
Let $A(z)$ be the minimal polynomial of a Pisot number $\theta$.  If $A(z)$ is not a reciprocal (and hence quadratic) polynomial, then $P_0(z) = \bigl(A(z) - A^*(z)\bigr)/(z-1)$ has degree at least $d-2$.
\end{lemma}
Thus, writing $A(z) = z^d + a_{d-1}z^{d-1} + \cdots + a_0$, this tells us that if $a_0=1$ then $a_{d-1}\ne a_{1}$.
In this sense, Pisot polynomials are strongly non-palindromic.
\begin{proof}
If $a_0\ne1$ then $P_0$ has degree $d-1$.
If $a_0=1$ and $a_{1}=a_{d-1}$, then expanding $A(z)/A^*(z)$ about $z=0$ gives
\[
A(z)/A^*(z) = 1 + u_2z^2 + u_3z^3 + \cdots\,.
\]
This contradicts \cite[Th\'{e}or\`{e}me 1]{DP} (which asserts that the coefficient of $z$ in such an expansion must be strictly positive), unless $A$ is a reciprocal quadratic polynomial.
\end{proof}

\subsection{A winding argument}\label{SS:Winding}

\begin{theorem}\label{T:winding}
Suppose that $A(z)$ is the minimal polynomial of a Pisot number.
For each $k\ge0$, define $P_k(z)$ as in Lemma \ref{L:Pk}.

For all large enough $k$, both $P_{k+1}$ and $P_k$ have all but two
roots on the unit circle, with the other roots real and positive,
and the roots of $P_{k+1}(z)$ and $(z-1)P_k(z)$ that lie on the unit
circle interlace.

For all $k\ge 1$, the rational function $(z-1)P_k(z)/P_{k+1}(z)$ is an interlacing quotient (either 
CC, CS or SS).
\end{theorem}

\begin{proof}
Suppose that $A(z)$ has degree $d$.
Note that $P_k(1)=kA(1)+A'(1)-(A^*)'(1)$, and this is negative for
all large enough $k$, since $A(1)<0$ ($A(z)$ is the minimal
polynomial of a Pisot number). Hence $P_k(z)$ has at least one
real root greater than 1, for all large enough $k$. Since $P_{k+1}$ and $P_k$
are both reciprocal, each has at least two positive real roots,
for all large enough $k$. From now on, we assume that $k$ is large
enough (say $k\ge k_0\ge 1$) for this to hold.

For $z$ on the unit circle, $(z-1)P_k(z)=0$ if and only if
$$z^kA(z)=A^*(z)=z^d A(\overline{z}) = z^d \overline{A(z)}\,,$$
which is equivalent to $z^{k-d}A(z)^2=|A(z)|^2$, which is
equivalent to $z^{k-d}A(z)^2$ being real and positive.

Now $z^{k-d}A(z)^2$ winds round the origin $k+d-2$ times as $z$
winds round 0.
Hence $(z-1)P_k(z)$ has at least $k+d-2$ zeros on the unit circle,
and $P_k(z)$ has at least $k+d-3$ roots on the unit circle.
Together with at least 2
roots not on the unit circle, we have accounted for all possible
roots: $P_k(z)$ has exactly $k+d-3$ roots on the unit circle, and
two other roots, real and positive (one of them being greater than
1 and the other between 0 and 1).

 Similarly $P_{k+1}(z)$ has   exactly $k+d-2$ roots on the unit
circle, and two other roots, real and positive.

For the interlacing property, we look more closely at what happens
as $z$ winds round 0 in the positive sense (anticlockwise), on the
unit circle, starting at $z=1$.
When $z=1$, the argument of $z^{k-d}A^2(z)$ is 0. As $z$ winds
round the unit circle, the argument increases to $(k+d-2)2\pi$,
not necessarily monotonically. The argument is an integer multiple
of $2\pi$ precisely when $P_k(z)=0$ (or when $z=1$). The argument
equals that of $1/z$ (modulo integer multiples of $2\pi$)
precisely when $P_{k+1}(z)=0$. It is clear from Figure \ref{fig1} that this
must happen at least once (and hence, by counting, exactly once)
between each two consecutive zeros of $(z-1)P_k(z)$, as claimed:
the line running from bottom left to top right (which need not be
a straight line!) must cross one of the short diagonals at least
once between each pair of horizontal lines.

\begin{figure}
\begin{picture}(450,130)(10,0)
\setlength{\unitlength}{0.8pt}
\multiput(50,30)(0,100){2}{\line(1,0){400}} \multiput(50,30)(100,0){5}{\line(0,1){100}} \multiput(50,130)(100,0){4}{\line(1,-1){100}} \put(50,30){\line(4,1){400}} \multiput(130,50)(80,20){4}{\circle*{5}} \multiput(150,55)(100,25){4}{\circle{5}}

\put(45,20){$0$} \put(145,20){$2\pi$} \put(245,20){$4\pi$} \put(345,20){$6\pi$} \put(445,20){$8\pi$} \put(170,15){\vector(1,0){140}} \put(200,0){$\arg\bigl(z^{k-d}A^2(z)\bigr)$}

\put(35,25){$0$} \put(30,125){$2\pi$} \put(10,80){$\arg z$} \put(40,60){\vector(0,1){40}}
\end{picture}
\caption{The case $d+k-2=4$. The zeros of $(z-1)P_k(z)$ [$\circ$] and of $P_{k+1}(z)$ [$\bullet$] interlace.}
\label{fig1}
\end{figure}
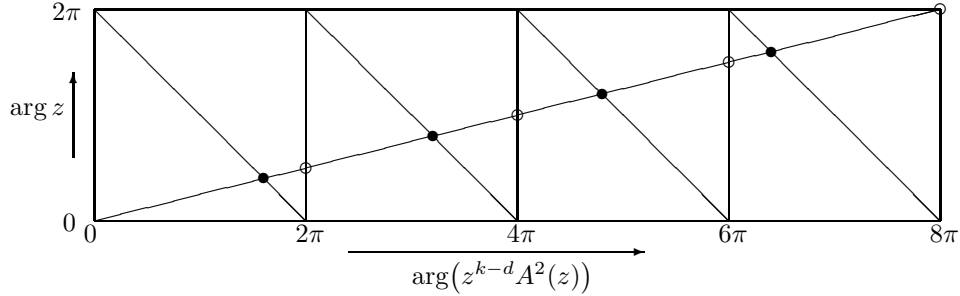

For the final assertion of the theorem, we need to consider what happens for smaller values of $k\ge 1$.
The winding argument still accounts for all but two of the zeros of each $P_k$  .  We need to pin down the other two roots, and establish the claimed interlacing property.
Let $\theta>1$ be the Pisot root of $A(z)$.
If $A^*(\theta)>0$, then $P_k(\theta)<0$, in which case $P_k$ always has a real root greater than $\theta$.
In this case   there is nothing more to prove: we have SS-interlacing.

We are left with the case that   $A^*(\theta)<0$.
Then $P_k$ has no real root greater than $\theta$, for any $k$.
In particular, for $k\ge k_0$ the Salem root $\tau_k$ of $P_k$ satisfies $1<\tau<\theta$.
Then from $z^kA(z)=P_{k+1}(z)-P_k(z)$ we have $P_{k+1}(\tau_k)<0$, and hence $\tau_{k+1}>\tau_k$.
If follows that $(z-1)P_k(z)/P_{k+1}(z)$ is a type $1$ SS-interlacing quotient for $k\ge k_0$, and the roots and poles of the corresponding real interlacing quotient $p_k(x)/p_{k+1}(x)$ interlace perfectly on the real line.
The recurrence (\ref{E:recurrence}) translates to the real world as
\[
p_{k+1}(x) = xp_k(x) - p_{k-1}(x)\,.
\]
Since the zeros of $p_{k_0}$ and $p_{k_0+1}$ interlace, one deduces that those of $p_{k_0-1}$ and $p_{k_0}$ interlace, and then those of $p_{k_0-2}$ and $p_{k_0-1}$, and so on.
Thus $P_{k+1}(z)$ and $(z-1)P_{k}(z)$ interlace for all $k\ge 1$  .
If $p_{k}(2)=0$, then $P_{k}$ has a double zero at $z=1$, and $(z-1)P_{k}$ has a triple zero there: in this case we have CS-interlacing.  If $P_{k+1}$ has all roots on the unit circle, then we have CC-interlacing (if $p_{k+1}(2)=0$, then $P_{k+1}$ has a double zero at $z=1$, and $P_{k+1}(z)/(z-1)$ interlaces with $P_{k}(z)$).  If $P_{k+1}$ is Salem but $P_{k}$ is cyclotomic, we have CS-interlacing.

\end{proof}

\subsection{Theorem \ref{T:SSPisot} gives all Pisot
numbers}\label{SS:allPisots}

Now we are in a position to show that the interlacing construction
given in Theorem \ref{T:SSPisot} produces all Pisot numbers, with
$h(z)=1/z$.

\begin{theorem}\label{T:allPisots}
Given any Pisot number $\theta$, there exists an SS-interlacing
quotient $Q(z)/P(z)$ satisfying the conditions of Theorem
\ref{T:SSPisot} (with $h(z)=1/z$), such that the only solutions to
$Q(z)/P(z)=1$ are $\theta$, its conjugates, and $0$.
\end{theorem}
\begin{proof}
Let $A(z)$ be the minimal polynomial of $\theta$. We consider
$P(z)=P_{k+1}(z)=\bigl(z^{k+1}A(z)-A^*(z)\bigr)/(z-1)$, and
$Q(z)=(z-1)P_k(z)=z^kA(z)-A^*(z)$.

We have seen (Theorem \ref{T:winding}) that for all large
enough $k$  the quotient $Q/P$ is an SS-interlacing quotient.  To
apply Theorem \ref{T:SSPisot} with $h(z)=1/z$ we need the
condition $\lim_{z\rightarrow 1+} Q(z)/(z-1)P(z)<1$. But
\[
\lim_{z\rightarrow 1+}\frac{Q(z)}{(z-1)P(z)} =\lim_{z\rightarrow
1+}\frac{P_k(z)}{P_{k+1}(z)} = \frac{kA(1)+A'(1)-(A^*)'(1)}
{(k+1)A(1)+A'(1)-(A^*)'(1)}\,,
\]
and this is less than $1$ if $k$ is large enough, since $A(1)<0$.

Finally we note that $Q(z)/(z-1)P(z)=1$ is equivalent to
$z^kA(z)=0$, which has as its roots $\theta$, all the conjugates
of $\theta$, and 0 (assuming $k>0$).
\end{proof}

For smaller values of $k$ the quotient $Q/P$ in the proof of Theorem \ref{T:allPisots} might be
CS-interlacing or CC-interlacing.
The case $k=0$ and $P_0(0)=0$ is exceptional, as ever.

The proof of Theorem \ref{T:allPisots} uses Salem's method to construct the $P_k$ from $A$, and then shows, conversely, how $A$ can be recovered from $P_k$ via (\ref{E:A}) of Lemma \ref{L:Pk}, 
 for $k$ sufficiently large.

\section{All Salem numbers via interlacing}\label{S:allSalems}
\subsection{Boyd's theorem}  We recall the following fundamental
result of Boyd \cite{Bo}.

\begin{theorem}[{{\cite[Theorem 4.1]{Bo}}}]\label{T:Boyd}
Let $\tau$ be a Salem number, with minimal polynomial $R(z)$. Define $S_1(z) = z^2 + 1$, $S_{-1}(z) = z-1$.  Then
for each choice of $\varepsilon=\pm1$ there exist infinitely many
Pisot polynomials $A(z)$ such that (with $A^*(z)$ as before)
\begin{equation}\label{E:Boyd}
S_\varepsilon(z)R(z) = zA(z) + \varepsilon A^*(z)\,.
\end{equation}
\end{theorem}

\subsection{All Salem numbers via interlacing}
Armed with Theorem \ref{T:Boyd}, we show first (Lemma \ref{L:allSalems}) that we can produce all Salem numbers
via SS-interlacing quotients, but with a ``right hand side'' other than $1+1/z$, as used 
in Theorems \ref{T:CCSalem}, \ref{T:CSSalem} and \ref{T:SSSalem}.

\begin{lemma}\label{L:allSalems}
Let $\tau$ be any Salem number, and choose $\varepsilon=\pm1$.
Then for all sufficiently large $k$ there exists an SS-interlacing
quotient $Q(z)/P(z)$ such that the only non-zero
solutions to
\begin{equation}\label{E:allSalems}
\frac{Q(z)}{(z-1)P(z)} =
\frac{z^{k-1}+\varepsilon}{z^k+\varepsilon}
\end{equation}
are $\tau$, its conjugates, and perhaps some roots of unity.
\end{lemma}

\begin{proof}
Let $R(z)$ be the minimal polynomial of $\tau$, and let $A(z)$ be
a Pisot polynomial such that (\ref{E:Boyd}) holds with our choice
of $\varepsilon$.  As in the proof of Theorem \ref{T:allPisots},
we put $P=P_{k+1}(z) = \bigl(z^{k+1}A(z) - A^*(z)\bigr)/(z-1)$ and $Q=(z-1)P_k(z) =
z^kA(z) - A^*(z)$, and repeat the observation that for all
sufficiently large $k$ the quotient $Q(z)/P(z)$ is an
SS-interlacing quotient.  With $z^kA(z) = P_{k+1}(z) - P_k(z)$ we have
$A^*(z) = P_{k+1}(z) - zP_k(z)$, and hence (with $S_\varepsilon(z) = z^2+1$ or
$z-1$ according as $\varepsilon=1$ or $-1$) from (\ref{E:Boyd}) we have
\[
\begin{array}{rcl}
z^{k-1}S_\varepsilon(z)R(z) &=& z^kA(z) + \varepsilon z^{k-1}A^*(z) \\
 & = & (P_{k+1}(z) - P_k(z)) + \varepsilon z^{k-1}(P_{k+1}(z) - zP_k(z)) \\
 & = & (1+\varepsilon z^{k-1})P_{k+1}(z) - (1+\varepsilon z^k)P_k(z)\,,
\end{array}
\]
from which the result follows.
\end{proof}

Instead of taking large $k$ in Lemma \ref{L:allSalems}, we can
consider choosing $k$ of any size. 
  The choice of $k=1$
gives the following
theorem.

\begin{theorem}\label{T:fourtypes}
Consider the equation
\begin{equation}\label{E:allSalems2}
\frac{Q(z)}{(z-1)P(z)} = \frac{2}{z+1}\,.
\end{equation}
Define four types of Salem number $I$, $II$, $III$, $IV$ as
follows.  A Salem number $\tau$ is of type $I$ (respectively,
$II$, $III$, $IV$) if there exist monic polynomials $P(z)$, $Q(z)$
such that $Q(z)/P(z)$ is a CC-interlacing quotient
(respectively, CS-interlacing, type $1$ SS-interlacing, type $2$
SS-interlacing) and for which the only non-zero solutions to
(\ref{E:allSalems2}) are $\tau$, its conjugates, and pehaps some
roots of unity.  Then every Salem number is of at least one of
these four types.
\end{theorem}

\begin{proof}

We take $k=1$  and $\varepsilon=1$ in the proof of Lemma \ref{L:allSalems}.
 For the interlacing properties, we appeal to Theorem \ref{T:winding}.
\end{proof}

\section{Comparison with the Bertin-Boyd classification}\label{S:BB}
Let $\tau$ be any Salem number, with minimal polynomial $R(z)$.
  Bertin and Boyd \cite{BB} showed
that there exist reciprocal polynomials $K(z)$   and
$L(z)$   such that $L(z)$ interlaces with
$K(z)R(z)$ on the unit circle.  Their Theorem B is most relevant
here, as it relates to expressing $K(z)R(z)$ in the shape $zA(z)
+\varepsilon A^*(z)$, where $A(z)=z^mA_0(z)$ is a Pisot
polynomial, with $A_0(0)\ne 0$.

In the case $\varepsilon=1$, which they use only when $A_0(0)<0$,
their polynomial $L(z)$ is $A(z) + A^*(z)$; in the case
$\varepsilon=-1$, which they use only when $A_0(0)>0$, their $L(z)$
is our $P_0(z)$.  Our proof of interlacing comes from a winding
argument; theirs is via a characterisation of ``entrances''
and ``exits'' to/from the unit disc for the associated algebraic
curve $zA(z) + \varepsilon t A^*(z)=0$ ($t\ge0$ real)---see
\cite[Lemma 3.1]{Bo}.

\section{Final remarks}\label{S:conclusion}
\subsection{Pisot numbers of negative
trace}\label{SS:negtracePisot} As one application of the
construction in Theorem \ref{T:CCPisot}, we produce an example of
a Pisot number that has trace $-1$ and degree only $16$.  Earlier
examples of Pisot numbers that had negative trace had much larger
degrees (\cite{MRS}, \cite{McK}).  The algorithm in \cite{MS2}
for producing Pisot numbers of any desired trace gives an example
with degree $38$.  The construction there was that in Theorem
\ref{T:CCPisot} with $h(z)=1/(z-1)$.  Instead, take
$g(z) = (z-1)(z^8 + z^7 - z^5 - z^4 - z^3 + z +
1)/(z^2-1)(z^3-1)(z^5-1)$, and $h(z) = z^7/(z-1)(z^7-1)$, to give
a Pisot number with degree $16$ and trace $-1$; its minimal polynomial is
\[
z^{16} + z^{15} - z^{14} - 4z^{13} - 6z^{12} - 7z^{11} - 7z^{10}
-7z^9 - 6z^8 - 4z^7 - 2z^6 - z^5 + z^3 + 2z^2 + 2z + 1\,.
\]

The choice of $g(z)$ (see \S\ref{SS:CC}) produces a low-degree example of a Salem number with trace $-1$ via Theorem \ref{T:CCSalem}.
The choice of the CC-limit function $h(z)$ is motivated by the desire to introduce a new, negative-trace, low-degree cyclotomic factor into the denominator.

We used the Dufresnoy-Pisot-Boyd algorithm \cite{Bo2} to search for small Pisot numbers of small degree and negative trace.
For Pisot numbers below $2$, we found $10$ examples, of degrees between $22$ and $48$.
The degree-$16$ example above is for a Pisot number slightly larger than $2$.
Finding the smallest degree (perhaps $16$?) for a Pisot number of negative trace remains a challenge.

\subsection{Salem numbers of large negative trace}
For Salem numbers of trace $-1$, see \cite{Sm}.
Salem numbers of trace below $-1$ first appeared via a graph construction \cite{MS1}, which can now be seen as a special case of Theorem \ref{T:CCSalem}.  In \cite{MS2}, Salem numbers of arbitrary trace were produced by interlacing, but the interlacing quotients were not optimal for producing minimal degrees.
Starting with the interlacing quotient $g(z) = (z-1)(z^8 + z^7 - z^5 - z^4 - z^3 + z +
1)/(z^2-1)(z^3-1)(z^5-1)$ from \S\ref{SS:CC}, add the CC-interlacing quotient $(z^{18}-1)/(z-1)(z^7-1)(z^{11}-1) + (z^{30}-1)/(z-1)(z^{13}-1)(z^{17}-1)$, and apply Theorem \ref{T:CCSalem} to produce a Salem number of degree $54$ and trace $-3$, the smallest degree currently known for this trace:
\[
z^{54} + 3z^{53} + 2z^{52} -11z^{51} -48 z^{50} -122z^{49} - 245z^{48} + \cdots\,.
\]
(One needs to check that this polynomial has no cyclotomic factors. This can be done using the algorithm of Beukers and Smyth \cite{BS}, or by checking irreducibility.)

A real transform of this polynomial can be found in \cite[\S5.4]{McK2}.


\subsection{Small Salem numbers}

\begin{proposition}
Let $\tau$ be any Salem number below the real root of $z^3 - z -
1$ (so that $\tau$ is smaller than any Pisot number).  Let $R(z)$
be the minimal polynomial of $\tau$, and let $A(z)$ be any Pisot
polynomial such that (\ref{E:Boyd}) holds with $\varepsilon=1$.
Then $A(z)$ has at least three real roots, with at least one
between $1/\tau$ and $1$.
\end{proposition}

\begin{proof}
The conditions on $\tau$ imply that $A(\tau)<0$.  Putting $z=\tau$
in (\ref{E:Boyd}) we deduce that $A^*(\tau)>0$ and hence $A^*$ has
a real root between $1$ and $\tau$.  Thus $A$ has at least two
real roots: the Pisot number, and another root between $1/\tau$
and $1$.  Since $A$ has odd degree it must have at least three
real roots.
\end{proof}

For example, taking $\tau$ to be Lehmer's number, one possibility
for $A(z)$ is $z^{11}
-2z^9-4z^8-4z^7-3z^6-z^5+z^4+3z^3+4z^2+3z+1$. Sure enough, this
has real roots approximately equal to $-0.74616$, $0.98390$,
$2.20974$.

It is not known whether or not there is a smallest Salem number.
If there is one, then the next theorem gives some information about it.

\begin{theorem}\label{T:towardsL}
If there is a smallest Salem number, $\tau$, then it is of type
$IV$ (as defined in Theorem \ref{T:fourtypes}), and not of any other type.
\end{theorem}

\begin{proof}
Suppose that there is a smallest Salem number, $\tau$.
We take $A(z)$ and the sequence $P_k(z)$ as in the proof of Lemma \ref{L:allSalems}, 
 and claim that  $Q(z)/P(z) = (z-1)P_1(z)/P_2(z)$ must be a type $2$ SS-interlacing quotient, showing that $\tau$ is of type IV.

If $Q(z)/P(z)$ were either CS-interlacing or type $1$ SS-interlacing, then the Salem root of $P(z)$ would be smaller than $\tau$, giving a contradiction.

Finally we eliminate the possibility that $Q(z)/P(z)$ is CC-interlacing.  In this case we increase $k$ until $P_{k+1}$ becomes Salem, with
$P_{k}$ still cyclotomic; then we get (using (\ref{E:allSalems}) with $Q/P=(z-1)P_k/P_{k+1}$, as in the proof of Lemma \ref{L:allSalems}) that the root of $P_{k+1}$ is
smaller than our Salem number, again contradicting the minimality of $\tau$.

\end{proof}

Note that this is not saying that there are no examples of small
Salem numbers that come from CC or CS interlacing: merely that if
we go via Boyd's theorem (as in the definition of types) we will not
see small Salem numbers arising other than as type IV.

\subsection{Further consequences of CC-interlacing}
We conclude with two amusing remarks concerning CC-interlacing quotients, which we record as a single proposition.

\begin{proposition}
Let $Q/P$ be a CC-interlacing quotient.  Then
\begin{itemize}
\item[(a)] $P^2+Q^2$ has all its roots in $\mathbb{T}$;

\item[(b)] $P+Q$ has all its roots in the open unit disc $|z|<1$.
\end{itemize}

\begin{proof}
For (a) we simply apply Lemma \ref{L:interlacingsum}(a): the sum $P/Q+Q/P$ is a CC-interlacing quotient, so its numerator has all roots in $\mathbb{T}$.  We can even say further that these roots interlace with those of $PQ$.

For (b) we use another winding argument.  Let $d$ be the common degree of $P$ and $Q$, and suppose that $z-1 \mid Q$.
We observe that it is enough to show that $f(z) = P(z^2) + Q(z^2)$ (a polynomial of degree $2d$) has all its roots in the open unit disc.

Write $f(z) = z^dg(z) = z^d \bigl( P(z^2)/z^d + Q(z^2)/z^d \bigr)$.
Since $P$ is reciprocal and $Q$ is antireciprocal, $P(z^2)/z^d$ and $Q(z^2)/z^d$ give the real and imaginary parts of $g(z)$ when $z$ is on the unit circle.
As $z$ goes round the unit circle, anticlockwise, the argument of $z^d$ increases by $2d\pi$.
The roots of $P$ and $Q$ interlace on the unit circle, so as $z$ goes round the unit circle the pair $\bigl(\Re g(z),\Im g(z)\bigr)$ cycles $d$ times through one of the patterns $(+,+)$, $(+,-)$, $(-,-)$, $(-,+)$ or $(+,+)$, $(-,+)$, $(-,-)$, $(+,-)$.  We do not (yet) know which of these patterns occurs, nor at what point in the pattern we start, but $d$ complete cycles through one of these two patterns must be made.  In either case, $g(z)$ winds $d$ times round the origin: in the former case it winds clockwise, and in the latter case anticlockwise.
We conclude that as $z$ goes anticlockwise around the unit circle, the argument of $f(z)$ increases by either $0$ or $4d\pi$.
It follows that $P(z^2) + Q(z^2)$ has either all of its roots in the open unit disc or none of them, and the same holds for $P(z) + Q(z)$.
But since $P$ is reciprocal and $Q$ is antireciprocal, we have $P(0)+Q(0) = 1 + (-1) = 0$, so that $P+Q$ has at least one root, namely $0$, that has modulus strictly less than $1$. Hence all $d$ roots must have modulus strictly less than $1$.

\end{proof}
\end{proposition}

\section*{Acknowledgment}

We are profoundly grateful to the referee who pointed out that in our first submitted draft we had been cavalier with our leading coefficients, leading us to assert several falsehoods.

\end{document}